\documentclass{amsart}
\usepackage{amssymb}
\usepackage{amsmath}
\usepackage{color}
\usepackage{graphicx}

\usepackage[pdftex]{hyperref}
\usepackage{dsfont}

\usepackage{algorithmicx}
\usepackage{algpseudocode}
\usepackage{tabularx}
\usepackage{algorithm}
\algnewcommand{\LeftComment}[1]{\Statex \(\triangleright\) #1}
\usepackage{amsmath, amssymb, MnSymbol}

\renewcommand{\theequation}{\thesection.\arabic{equation}}
\newtheorem{theo}{\indent Theorem}[section]
\newtheorem{prop}[theo]{\indent Proposition}

\newtheorem{rem}[theo]{\indent Remark}
\newtheorem{lem}[theo]{\indent Lemma}
\newtheorem{defin}[theo]{\indent Definition}
\newtheorem{cor}[theo]{\indent Corollary}

\newtheorem{ass}[theo]{\indent Assumption}

\newtheorem{algo}{Algorithm}

\def\D{{\mathbb D}}

\def \E{I\!\!E}

\newcommand{\R}{\mathbb {R}}

\newlength{\breite}
\title[Hawkes]{Perfect simulation for interacting Hawkes processes with variable length memory}
\date{}
\author{Branda P. I. Goncalves and Paul Gresland}

\address{B. P. I. Goncalves: Universit\'e Paris Nanterre, 200 avenue de la république, 92001 Nanterre Cedex, France.}

\email{b.goncalves@parisnanterre.fr}

\address{P. Gresland: Universit\'e Côte d'Azur, 28 Avenue de Valrose, 06103 Nice CEDEX 2, France}
\email{paul.gresland@gmail.com}

\begin{document}
\maketitle
\def\abstractname{Abstract}
\begin{abstract}

We consider a nonlinear multivariate Hawkes process having a variable length memory which allows to describe the activity of a neuronal network by its membrane potential.  We propose a graphical construction of the process and we construct, by means of a perfect simulation algorithm, a stationary version of the process.
By making the hypothesis that the spiking rate $\beta_i$ of the neuron $i \in I $ is bounded, we construct an algorithm based on a priori realizations of the Poisson process $(M^i, i \in I)$. We show that there exists a critical value $\delta_c$ such that if $\underline{\delta} > \delta_c$ (where $\underline{\delta}= \inf_i{\delta_i}$ with $\delta_i = \frac{\beta_{i* }}{\beta^*_i-\beta_{i*}}$ ) the process is ergodic.

\end{abstract}

{\it Key words} : Nonlinear multivariate Hawkes process, process with variable length memory, neuronal network, membrane potential, perfect simulation.
\\

{\it MSC 2000}  : 60 J 55, 60 J 35, 60 F 10, 62 M 05

\section{Introduction}

We consider nonlinear multivariate Hawkes processes (which are close to the process introduced by Hawkes \cite{Hawkes}) with transition probabilities that depend in the last spiking time in the past and not the whole past, that is, Hawkes processes with variable length memory. 

Hawkes processes are popular, as they allow to model in a relevant way phenomena in various domains such as finance, seismology, neuroscience. Multivariate Hawkes processes have long been studied in probability theory see,  Daley and Vere-Jones \cite{DV-J}, Brémaud and Massoulié \cite{BM}, Massoulié \cite{Mas}, etc...

The model considered in this paper is an extension of the perfect simulation part of model presented in Goncalves \cite{Gon}.
We use here, Hawkes processes having a variable length memory to describe the activity of each neuron $i \in I$, where  $I \subset \mathbb{Z}$ is a subset of $\mathbb{Z}$. In this model, each neuron $i\in I$ sends its spike at the random intensity $\lambda_t^i= \beta_i(X_{t_-}^i)$ where $\beta_i$ is the jump rate function of neuron $i$ and $ X_t^{i}$ is the membrane potential of neuron $i$ at time $t$. \\
 
In the literature, graphical methods (based on the perfect simulation algorithm) are used to construct processes of infinite size in their stationary regime, see Comets et al. \cite{CFF} and Galves et al. \cite{GLO}.
In Comets et al. \cite{CFF}, the authors consider processes with transition probabilities that depend on the whole past history and present a perfect simulation algorithm for stationary processes indexed by $\mathbb{Z}$ with summable memory decay. The authors generalize in Delattre et al. \cite{DFH}, the construction of multivariate Hawkes processes to a possibly infinite network of counting processes on a directed graph without giving an explicit construction of the process.
In the work of Ferrari et al. \cite{FGGL}, considering an infinite system of interacting point processes with memory of variable length, the authors investigated the conditions for the existence of a phase transition using the classical contour technique, based on the classical work of Griffeath \cite{Gri} on a contact process. This condition of existence of a phase transition is also given in Goncalves \cite{Gon}, where the neural network is purely inhibitory and the neurons are represented by their inhibition state. We construct in \cite{Gon} a perfect simulation algorithm to show the recurrence of the process under certain conditions in the Markovian case.

Contrary to the models studied in Ferrari et al. \cite{FGGL} and Goncalves \cite{Gon} where we remain in a Markovian framework, we are going to apply the same graphical construction techniques (using perfect simulation algorithms) on nonlinear multivariate Hawkes processes in their stationary regime. \\

The papers is organized as follows. In Section 2, we describe the model and establish the assumptions we need. In Section 3, we present the construction of the perfect simulation algorithm and main result, Theorem \ref{theo:val critique},  which gives the condition under which the algorithm stops after a finite number of steps. An example applied to the model is presented in Section 4.

\section{The model}

\subsection{ Description of the model }

In our paper, we consider a countable set of  interacting neurons $I$.
For all $i \in I,$ let $N^i$ be the counting process of successive spikes of neuron $i$, that is, for all $0< s < t < \infty, \ \ N^i(]s,t]) $ counts the number of spikes of neuron $i$ during the interval $]s,t].$
We define for all $i \in I,$ $X^{i}_t$ as describing the  membrane potential (that is, electrical potential difference between the inside and the outside of the cell) of the $i-$th neuron. $X^{i}_t$ can thus be represented as a solution of the equation
\begin{equation}\label{eq:1}
X_t^{i}= \sum_j \int_{]L_t^{i},t]}  h_{ij}(t-s) dN^j_s, \ \ i\in I
 \end{equation} 
where  $ h_{ij} (t) $ 
is a family of synaptic weight functions modeling the influence of neuron $j$ on neuron $i$ and
$$ L_{t}^{i}=\sup \left\{s<t: \Delta N^{i}_s = 1\right\} $$
is the last spiking time before time $t$ of neuron $i$ and $\Delta N_s:= N_s-N_{s-}$. 
We can interpret our model as follows: when a neuron $j$ spikes, it sends a synaptic weight to neurons $i$ at time $s \in (L_t^i, t]$, this spike modifies the membrane potential $ X_t^i$ of neurons $i$ that are in the neighborhood of neuron $j$ and the membrane potential of neuron $j$ is reset to $0$ at time $s$. \\

\vspace{0.2cm}

The process $(N_t^1,\cdots, N_t^N)_{t>0}$ is then a Hawkes multivariate process with intensity 
\begin{equation} \label{eq:intensity}
\lambda_t^i= \beta_i \left(\sum_{j\to i} \int_{]L_t^{i},t]}  h_{ij}(t-s) dN^j_s\right), \ \ \forall i \in I, 
\end{equation}
where $\beta_{i}: \mathbb{R}_+ \to \mathbb{R}_{+}$ is the spiking rate function. This intensity process given by \eqref{eq:intensity} is close to the typical form of the intensity of a multivariate nonlinear Hawkes process. The only difference is that here we are only interested in the last spiking time in the past and not the whole past.  
\vspace{0.2cm}

%
%
 
\begin{ass} \label{eq:beta}
The spiking rate functions $\beta_{i}: \mathbb{R}_+ \to \mathbb{R}_{+}$ are decreasing and bounded on $\mathbb{R}_+$, that is, for each $x >0,$ for all $i \in I, $ 
\begin{equation*} 
0 < \beta_{i*} < \beta_i(x) < \beta_i^{*} < \infty . 
\end{equation*} 

\end{ass}

\begin{rem}
The formula (\ref{eq:1}) is well-posed in the sense that there is non explosion of the process.   Since $ \beta_i ( X_s^{i} ) \leq \beta_i(0) $ for all $i$ we deduce that $ \int_0^t \lambda_s^i ds < \infty $ whence the non explosion, that is, almost surely, the process has only a finite number of jumps within each finite time interval.
 \end{rem}

%

\vspace{0.2cm}

For each neuron $i$, let  $\mathcal{V}_{. \to i } = \{ j \in I: W_{j \to i} \neq 0\} $ be the incoming neighborhood of neuron $i$, that is, the set of neurons that have a direct influence on neuron $i$  and $\mathcal{V}_{i \to . } = \{ j: W_{i \to j} \neq 0\} $ the out-coming neighborhood, that is, the set of all neurons that are directly influenced by neuron $i$  (see Comets et al. \cite{CFF},  Galves and L\"{o}cherbach \cite{GL} and Goncalves \cite{Gon}). \vspace{0.2cm}


%
%

\section{Perfect simulation}

In this section, we want to construct the stationary nonlinear Hawkes process by a perfect simulation procedure based on an a priori realization of a Poisson process $(M^i, i \in I)$ of intensity $\beta^*_i$. 
We use the idea developed in Hodara and Löcherbach \cite{HL} about a dominating Poisson random measure.
For this, we suppose assumption \eqref{eq:beta} holds and we introduce a sequence of i.i.d Poisson random measure $M^i(dt, dz)$ on $\mathbb{R}_+ \times \mathbb{R}_+$ of intensity $dt dz$ dominating the process $(N^i, i\in I)$.

\begin{defin} \label{def:hawkes}
A family $\left(N^{i}, i \in I\right)$ of random point measures defined on a probability space is said to be a Hawkes process with variable length memory  with parameters $(\beta, h)$ if almost surely, for all $i \in I$, 
\begin{equation} \label{def:N}
N^{i}_t= \int_{]0,t]} \int_{[0,\infty[} 1_{\left\{z \leq  {\beta_i}\left(\sum_{j} \left(\int_{]L_{t}^{i}, t]} h_{ij} (t-s) d N_{s}^{j}\right)\right)\right\}} M^i(dt, d z)
\end{equation}
\end{defin}
According to Brémaud and Massoulié \cite{BM}, see also Proposition 3 of Delattre et al. \cite{DFH} and Hodara and Löcherbach \cite{HL} a Hawkes process according to Definition \ref{def:hawkes} is a Hawkes process according to \eqref{eq:intensity} and vice versa.

\vspace{0.2cm}

Formula \eqref{def:N} implies that we can construct the process $\left(N^{i}, i \in I\right)$ by a thinning procedure applied to the a priori family of dominating Poisson random measures $M^{i}(dt, dz)$ having intensity $ d t dz$. Since $N^{i}$ is a simple point measure, it is enough to define it through the times of its atoms. Each atom of $N^{i}$ must also be an atom of $M^{i}$ since $N^{i} \ll M^{i}$. 
\vspace{0.2cm}

 We decompose the Poisson process $M^i $ of intensity $\beta^*_i$ as $$ M^i= M^{i,s}+ M^{i,p}, $$ where $M^{i,s}$ and  $M^{i,p}$ are independent Poisson processes with respective intensities $\beta_{i*} $ and $\beta_i^*-\beta_{i*}$. \\
Notice that, all jumps time of process $M^{i,s} $ are the same that the process $N^i$ conditionally on the processes $ M^{i,s}$ and $M^{i,p}$. We call them the {\it sure jumps.} 
They appear at a jump time $T^{i,s}$ of $M^{i,s}$ with probability $p_0 := \frac{\beta_{i*}}{\beta^*_i}.$ Moreover, any jump time $T^{i,p} $ of $M^{i,p} $ will be a jump time of $N^i$ with probability 
$$  p_i:= \frac{\beta_i \left(\sum_{j\to i} \sum_{T_n^j \in]L^i_t,t]}  h_{ij}(t-T_n^j) \right) -\beta_{i*}}{\beta_i^*-\beta_{i*}}, $$ 
and  we have to decide for each neuron $i\in I$ and each time $T^{i,p} $ whether this jump is accepted or not. \\

To construct this stationary nonlinear Hawkes process, we fix a neuron $i \in I $ and in what follows we are interested in finding the  membrane potential of neuron $i$ in its stationary regime. 
The variable $T:=(T_1, T_2, \cdots)$ which is the time vector will be used to write the perfect simulation algorithm. 

\begin{defin}
Let $i\in I$ and $t\in \mathbb{R}_+$. The clan of ancestors $C_t^i$ of neuron $i$ at time $t$ is the set of all the neurons $j\in I $ that might influence the neuron $i$. It evolves in time by successive jumps.
\end{defin}

We define $$ T_{next}= \sup\{ s < t : X_s^{i}=0 | X_t^{i}=0 \} \text{ and }  \partial_{ext}(C_t^i) = \{ j \notin C_t^i : \exists k\in C_t^i,  h_{kj} >0 \} $$ where $T_{next}$ is the next jump time in the clan of ancestors of neuron $i$ after time $t$ and $\partial_{ext}(C_t^i)$ is the set of neurons not belonging to the clan of ancestors  of neuron $i $ but having an interaction with at least one neuron in the clan of ancestors  of neuron $i.$\\

\begin{algo}{Backward procedure}
\begin{enumerate}
\item We simulate , for each $ j \in I, \;\; M^{j,s} \text{ and } M^{j,p} $ two Poisson processes with respective intensities $\beta_{i*} $ and $\beta_i^*-\beta_{i*}. $ The jump times of $ M^{j,s} $ and $M^{j,p} $ are respectively $T_n^{j,s} $ and $T_n^{j,p} $ for the neuron $j $ after $n $ jumps.\\
\item Initialize the family $ \mathcal{V}_{. \to i} $ of non empty incoming neighborhoods of the neuron $i$\\
\item Initialize  $C_{T_0}^i =\{ i \} $ the clan of ancestors 
of neuron $i$ at time $T_0. $   For all time $t <T_0,$ we let $C_t^i $ the clan of ancestors of neuron $i$ at time $t$\\ 
\item  We set $V(C_t^i) := \cup_{j \in C_t^i} \mathcal{V}_{. \to j} $. While $ | C_t^i | >0 $ (where $| C_t^i| $ denotes the cardinality of $C_t^i$) do : \\

 -Determine the next jump time $ T_{next} < t $ in the clan of ancestors of neuron $i $ at time $T_{next} $ and in $ \partial_{ext} (clan) $, the correspondant neuron $j $ and the nature of jump \vspace{0.1cm}
 
- If neuron $j \in C_t^i $ and the jump is sure, i.e, $ T_{next} = T^{j,s},$  then,  $C_{T_{next}}^i = C_t^i \setminus \{i\} $ \vspace{0.1cm}
 
- If $j \in  C_t^i $ and the jump is possible, i.e, $ T_{next} = T^{j,p},$ then, $C_{T_{next}}^i = C_t^i $ \vspace{0.1cm}
 
- If $j \in V(C_t^i) $ and the jump is sure,  i.e, $ T_{next} = T^{j,s},$ then, $C_{T_{next}}^i = C_t^i $ \vspace{0.1cm}
 
- If $j \in V(C_t^i) $ and the jump is possible,  i.e, $ T_{next} = T^{j,p},$ then, $C_{T_{next}}^i = C_t^i \cup \{j\} $ \vspace{0.1cm}
 
- We update $ t \leftarrow T_{next}$ \vspace{0.1cm}

\item[] end While.
\end{enumerate}
\end{algo}

We stop this Algorithm at time $T^i_{stop}:= \inf\{t : C_t^i = \emptyset\}$ which is the first time when the clan of ancestors of neuron $i$ is empty. Indeed, the whole procedure makes sense only if $ T_{stop} ^i < \infty $ almost surely.
\vspace{0.2cm}

In the following we will write a forward procedure of the process in the case $ T_{stop}^i < \infty .$  \\

For this, we define: $$ N_{stop}^i = \inf\{ n >0: C_{T_n}^i= \emptyset \}, \    \bar{\mathcal{C}}^i = \cup_{n=0}^{{N}^i_{stop}} C_{T_n}^i  $$ 
where $ N_{stop}^i $ is the number of steps of the backward procedure and $ \bar{\mathcal{C}}^i  $ is the union of all clans of ancestors up to $N^i_{stop}$.

We denote by 
$$\mathcal{D}^i:=  \bar{\mathcal{C}}^i  \cup \partial_{ext}(C_t^i)  $$ 
the set of neurons which belong to a clan of ancestors of neuron $i $ at a time $t$ or to its neighborhood.

In this algorithm, we will rely on the a priori realizations of the processes  $  M_t^{i,s} \text{ and } M_t^{i,p}.  $ 
We can realize the acceptance/rejection procedure of the elements in the clans of ancestors already determined in the first algorithm. \\
We start with the positions for which the jumps are sure. Then, we keep in mind all the jumps that we accept regardless of everything. During the algorithm, we will gradually update all the sure positions of the jumps.

\begin{algo}{Forward procedure}
\begin{enumerate}

\item We determine the chronological list of the different jump times of the processes $  M_t^{i,s} \text{ and } M_t^{i,p} $ from $T_0$ to the last time which makes the clan empty.  

- For each of these jump times, we indicate the associated neuron and the nature of the jump, i.e., at time $T_n$ we have $T_n = T_n^{j,l} $ where $l \in \{ s,p\}$ and $j$ is the associated neuron. 

- If the jump is sure, i.e, $T_n = T_n^{j,s} $,  the position $x^j_{T_n}$ of the neuron $ j $ associated with this jump time is 0. \\

\item We set $m \gg 1$. While $m >0 $ do 

- Let $m $ be the rank of the last possible jump time $T_m $ of $\mathcal{D}^i$  in the chronology of jump times. 

- Let $k $ be the neuron associated with this jump, i.e.,  $T_m = T_m^{k,p}$.\\

\item[] We determine the rank $r $ of the last sure jump time $T_r =T_r^{k,s} <  T_m^{k,p} = T_m $ of $k$ in the chronology of jump times.  The state $x_{T_m}^k$ of neuron $k$ at time $T_m$ is determined as follows:

- We set $ x_{T_{r}}^k = 0$. 
$$ A^k_m : = \sum_{j}  \sum \limits_{\substack{ T_{m+l}^j \in (T_r , T_m), \\  l  \in \{1 ,\cdots, r-m-1\}}} h_{kj}(T_{m} - T_{m+l}^j )  \mathds{1}_{\{\text{ $ T_{m+l}^j = T_{m+l}^{j,s}$} \} } 
$$

- We determine if the occurence is effective or not for jump $k$ at time $T_m$ thanks to $A^k$. We have:
$$ x_{T_m}^k= A^k_m \mathcal{B}(1-p_k),  \text{ with } p_k= \frac{\beta_k(A^k_m) - \beta_{k*}}{\beta_k^*-\beta_{k*}}  $$
 where $\mathcal{B}(p)$ is the Bernoulli distribution of parameter $p$.
 Then, if the jump is effective, $ x_{T_m}^k= 0. $
 Repeat  step (7) of the procedure. 
 
\item[] end While.

\begin{rem}
After this step, we know the exact nature of all jumps.
\end{rem}

\item Determine for neuron $i$ its first sure jump time $T_n= T_n^{i,s}$ where $n$ is the rank of this time in the chronology of jump times.\\

\item The state of neuron $i$ at time $T_0 $ is determined by:  
$$ x^i_{T_0} = \sum_{j} \sum\limits_{\substack{T_K^j \in ( T_n, T_0), \\ K<n}} h_{ij}(T_0-T_K^j)   $$

\begin{rem}
The last value determined is the potential of neuron $i$ at the time $T_0$ in its stationary state.
\end{rem}
\end{enumerate}
\end{algo}

\begin{rem}
This algorithm 2 is inspired by \cite{GGLO} (page 20-21) which shows that if we find the potential of a fixed neuron $i \in \mathbb{Z}$  at time $T_0$, it is necessarily the potential of the neuron $i$ in its stationary regime. The algorithm is not a proof in itself, but allows to have an idea of the theoretical distribution of the value of neuron $i$ at time $T_0$ in its stationary regime.
\end{rem} 

\begin{prop} \label{exchange_prop}
By exchangeability argument \footnote{ We say that $N$ particles are exchangeable, if the law of $(X_t^1, \cdots, X_t^N)$ is stable under the action of a coordinate permutation.}, for any neuron $i \in \mathbb{Z}$, the law of the stationary process $\mathbf{P}$ is such that
\begin{equation*}
\mathbf{P}(X_t^{i} = 0) = \mathbf{P}(\text{last jump before time } 0 \text{ of particles } i, i - 1, i+1 \text{ is a jump of } i) = \frac{1}{3}. 
\end{equation*} 
\end{prop}

\begin{proof}
The neuron $ i $ interacts only with its two nearest neighbors ($i-1$ and $i+1$). Thus, by exchangeability the probability that its potential is $0$ is the probability that it jumps last or $1/3$.
\end{proof}

\begin{theo} \label{theo:val critique}
 We set $\bar \delta= \sup_i \delta_i$ and $\underline{\delta} = \inf_i \delta_i $ where $\delta_i=\frac{\beta_{i* }}{\beta^*_i-\beta_{i*}}. $
There exists a critical value $0<\delta_c < \infty  $ such that:
\begin{enumerate}
\item[-] if $\bar{\delta} < \delta_c, $ then the extinction time is infinite with a positive probability that is, $ \mathbb{P}(\forall i, \ T_{stop}^i = \infty) >0. $
\item[-] if $\underline{\delta } > \delta_c, $ then the extinction time is finite almost surely that is, $ \mathbb{P}( \forall i, \ T_{stop}^i < \infty)=1 $
\end{enumerate}
\end{theo}

\begin{proof}
The proof of this theorem is almost the same as that of Theorem 10 of \cite{Gon}. 
In the first part it will be necessary to replace $\delta$ by $\bar{\delta}$. \\
In the second part, with a rate $ n $ the transition of the branching process $ Z_t $ is from $ n $ to $ n+1 $ and with a rate $ n \underline{\delta} $ this transition is from $ n $ to $ n-1. $ 
The associated infinitesimal generator of  $Z_t$ for any bounded test function $f$ is: $$Af(n)= n [(f(n+1)-f(n))+ \underline{\delta} (f(n-1)-f(n))].$$

\end{proof}

\section{Example}

 \begin{ass}

1) We consider that any neuron $i$ interacts with its nearest neighbors, $i-1$ and $i+1$. The non-zero synaptic weights are such that, for any $(i,j) \in \mathds{Z}^{2}$,
      \begin{equation*} 
         h_{i j}(t) = W \times \mathds{1}_{|i-j| = 1}  h(t)
    \end{equation*} 
    where $W$ is a positive constant and $h$ is a positive function.
    
 2)   The spiking rate functions $\beta: \mathbb{R}_+ \to \mathbb{R}_{+}$ is the same for all neurons. It is decreasing and bounded on $\mathbb{R}_+$, that is, for any $x >0$,
    
    \begin{equation*} 
        0 < \beta_{*} < \beta(x) < \beta^{*} < \infty . 
    \end{equation*} 
\end{ass}

\subsection{Algorithm description}
    
 The algorithm code is object-oriented with two classes, one called neuron and the other called jump. The following table provides the data structure associated with the jumps object class.
    
    \renewcommand{\arraystretch}{1.1}
        \newcolumntype{M}[1]{>{\centering\arraybackslash}m{#1}}
        \begin{table}[h]
        \label{jump_table}
            \begin{center}
                \begin{tabular}{|M{4cm}|M{12cm}|}
                    \hline
                    Component & Meaning 
                    \tabularnewline
                    \hline
                    $j \in \mathds{N^{*}}$ & Index of the jump.
                    \tabularnewline
                    \hline
                    $i \in \mathds{Z}$ & Index of the neuron that produced the jump
                    \tabularnewline
                    \hline
                    $t_{j} \in \mathds{R_{-}}$ & Time occurrence of the jump.
                    \tabularnewline
                    \hline
                    $U_{j} \in [0,1]$ & Uniform random variable 
                    \tabularnewline
                    \hline
                    $\mathds{1}_{j} \in \left\{0,1\right\}$ & Indicate if the jump is sure
                    \tabularnewline
                    \hline
                 \end{tabular}
            \end{center}
            \caption{Data structure of the jumps object class} 
        \end{table}
        
    With these notations we can go into details of the algorithm. The code successively performs the backward procedure \ref{alg: backward} and the forward procedure \ref{alg: forward}. 
 
  \clearpage 
   
    \begin{algorithm}
        \caption{Backward procedure} 
        \label{alg: backward}    
        \begin{algorithmic}[1]
        
            \LeftComment{\textbf{Initialization}}
            \State Initialize $J$, the list of jumps as $\emptyset$
            \State Initialize $C$, the clan of ancestors as a list that only includes neuron $0$
            \State Initialize $S$, simulated neurons as a list that includes neuron $0$ and its neighbors $-1$ and $1$ 
            \State Initialize the number of jumps $j = 0$ 
            \State Initialize the time $T_{0} = 0$ \vspace{0.2cm}
            
            \LeftComment{\textbf{Building the list of jumps $J$ that impacted the state of the neuron $0$ at time $T_0$}}  
 \vspace{0.1cm}
            \While{$C \ne \emptyset$}
                \State $j = j + 1$ 
                \State Crate a new jump with index $j$ 
                \State Draw the neuron $i$ that produced the jump, uniformly picked among $S$
                \State Draw an exponential random variable to determine the jump time, $t_{j} = t_{j-1} - \mathcal{Exp}(|S| \times \beta^{*})$
                \State Draw the uniform random variable $U_{j} \in [0,1]$
                \If{$U_{j} < \frac{\beta_{*}}{\beta^{*}}$}
                    \State The jump is sure, $\mathds{1}_{j} = 1$
                    \If{$\left\{ \text{neuron } i \right\} \subset C$, neuron $i$ is in the clan of ancestors,}
                        \State the neuron $i$ is removed from the clan of ancestors, $C = C \setminus \left\{ \text{neuron } i \right\}$
                        \If{$\left\{ \text{neuron } i-1 \right\} \not\subset C$ \textbf{and} $\left\{ \text{neuron } i-2 \right\} \not\subset C$}
                            \State The neuron $i-1$ is removed from the simulated neurons, $S = S \setminus \left\{ \text{neuron } i-1 \right\}$
                        \EndIf  \vspace{0.1cm}
                        \If{$\left\{ \text{neuron } i+1 \right\} \not\subset C$ \textbf{and} $\left\{ \text{neuron } i+2 \right\} \not\subset C$}
                            \State The neuron $i+1$ is removed from the simulated neurons, $S = S \setminus \left\{ \text{neuron } i+1 \right\}$
                        \EndIf  \vspace{0.1cm}
                    \EndIf  \vspace{0.1cm}
                \Else{ ($U_{j} \geq \frac{\beta_{*}}{\beta^{*}}$) \textbf{then}}
                    \State The jump is a candidate jump, $\mathds{1}_{j} = 0$
                    \If{$\left\{ \text{neuron } i \right\} \not\subset C$, neuron $i$ is not yet in the clan of ancestors,}
                        \State the neuron $i$ is append to the clan of ancestors, $C = C \cup \left\{ \text{neuron } i \right\}$
                        \If{$\left\{ \text{neuron } i \right\} \not\subset S$, neuron $i$ is not yet simulated,}
                            \State the neuron $i$ is append to the simulated neurons, $S = S \cup \left\{ \text{neuron } i \right\}$
                        \EndIf
                        \If{$\left\{ \text{neuron } i-1 \right\} \not\subset S$, neuron $i-1$ is not yet simulated,}
                            \State the neuron $i-1$ is append to the simulated neurons, $S = S \cup \left\{ \text{neuron } i-1 \right\}$
                        \EndIf  \vspace{0.1cm}
                        \If{$\left\{ \text{neuron } i+1 \right\} \not\subset S$, neuron $i+1$ is not yet simulated,}
                            \State the neuron $i+1$ is append to the simulated neurons, $S = S \cup \left\{ \text{neuron } i+1 \right\}$
                        \EndIf  \vspace{0.1cm}
                    \EndIf  \vspace{0.1cm}
                \EndIf 
                \State The jump $j$ is append to $J$ with its attributes $i$, $t_{j}$, $U_{j}$, $\mathds{1}_{j}$,  
            \EndWhile
        \end{algorithmic} 
    \end{algorithm}
    
\begin{algorithm}
        \caption{Forward procedure} 
        \label{alg: forward}    
        \begin{algorithmic}[1]
        
            \LeftComment{\textbf{Initialization}}
            \State $J$ is the list of jumps built during the backward phase and sorted in increasing time order
            \State $S$, is a list that includes neuron $0$ and all the other neurons that has at least one jump in $J$
            \For{each neuron $i \in S$}
                \State Initialize its list of its pre-synaptic jumps, $P_{i}$ = $\emptyset$ 
            \EndFor
            \LeftComment{\textbf{Compute the state of the neuron $0$ at time $T_0$}}
            \For{each jump $j \in J$}
                \State Identify $i$, the index of the neuron that produced the jump $j$
                \If {$\mathds{1}_{j} = 0$, the jump $j$ is not sure}
                    \State Compute the state $x^{i}_{t_{j}}$ of the neuron $i$ at time $t_{j}$, $x^{i}_{t_{j}} = \sum_{j'\in P_{i} } W \times h(t_{j} - t_{j'})$
                    \If {$\frac{\beta(x)}{\beta^{*}} \geq U_{j}$, the firing rate is enough to produce the jump $j$,}
                        \State Set the jump $j$ as sure, $\mathds{1}_{j} = 1$
                    \EndIf 
                \EndIf 
                \If {$\mathds{1}_{j} = 1$, the jump $j$ is sure}
                    \State The neuron $i$ reset, erase all its pre-synaptic jumps, $P_{i} = \emptyset$
                    \State The neuron $i-1$ receives a pre-synaptic spike, append the jump $j$ to $P_{i-1}$
                    \State The neuron $i+1$ receives a pre-synaptic spike, append the jump $j$ to $P_{i+1}$
                \EndIf
            \EndFor
            \State Compute the state $x^{0}_{T_0}$ of the neuron $0$ at time $T_0$, $x^{0}_{T_0} = \sum_{j'\in P_{0} } W \times h(T_0- t_{j'})$

        \end{algorithmic} 
\end{algorithm}  
  
    
\subsection{Results}
    
    We simulate, with the algorithm described above $N=100000$ values for the inhibition state. We then estimate non parametrically the distribution of the inhibition state Figure \ref{fig: empirical_dist}.

    \begin{figure}[h!]
        \centering
        \includegraphics[scale=0.52, height=4.5cm, width=7cm]{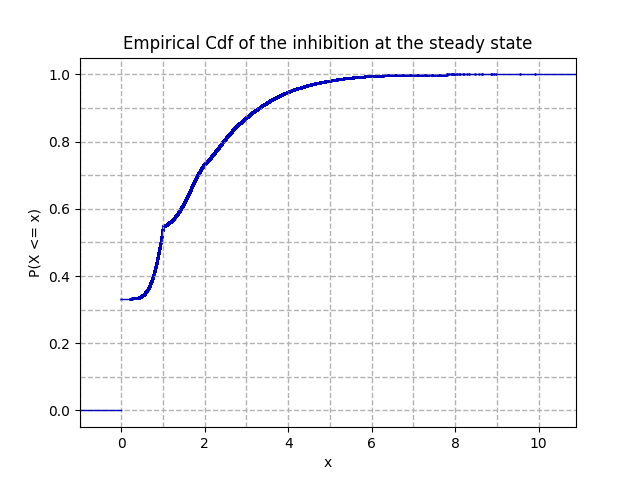}
        \includegraphics[scale=0.52, height=4.5cm, width=7cm]{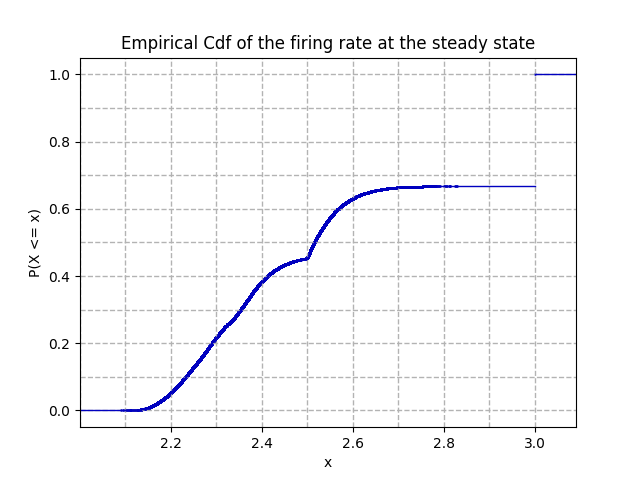}
        \caption{The function $\beta(x) = \frac{\beta^{*}-x\beta_{*}}{x+1}$ returns the firing rate (distribution on the right-hand graph) from the inhibition state (distribution on the left-hand graph). $\beta(x)$ is always between $\beta_{*}$ and $\beta^{*}$ here equal to $2Hz$ and $3Hz$. A firing rate equals to $\beta^{*}$ corresponds to an inhibition state equals to $0$ and a firing rate equals to $\beta_{*}$ corresponds to an inhibition state equals to $+\infty$. The maximum inhibition observed in the $N=100000$ replicates is about $10$. This corresponds to at least as much presynaptic spikes since the value $W$ of the synaptic weights is set to $1$.}
        \label{fig: empirical_dist}
    \end{figure}

    The distribution has irregularities, especially for low inhibition values. We assume that the irregularities arise from the number of presynaptic jumps that the neuron received between its last reset and the time $0$. The distribution of the number of presynaptic jump received Figure \ref{fig: n_jumps} is obviously discrete. However, when the number of pre-synaptic jumps is large, the integral against the interaction function $h$ provides regularity. The distributions shown in Figure \ref{fig: empirical_dist_lamb} are established with the interaction function $h(t) = (1+t)^{-\lambda}$ where the parameter $\lambda = 2$. We can reduce the regularisation effect by using a smaller values for $\lambda$. \\
    
    \begin{figure}[h!]
        \centering
        \includegraphics[scale=0.52, height=5cm, width=7cm]{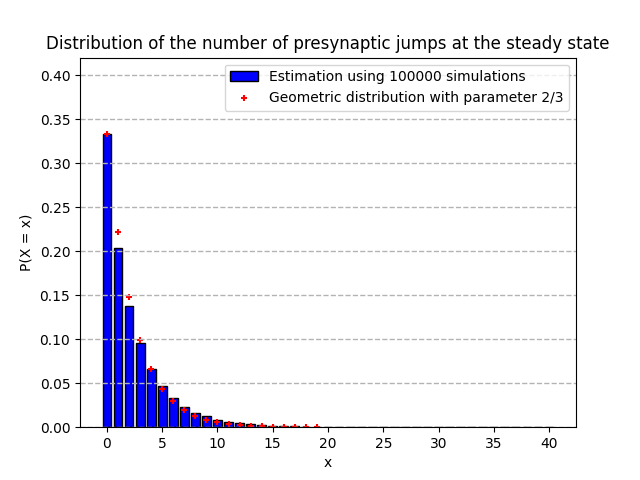}
        \caption{For small values, the number of presynaptic spikes is close to a geometric distribution with parameter $\frac{2}{3}$. In particular, the probability that the neuron $0$ has not received any presynaptic spikes at time $0$ is $\frac{1}{3}$ as predicted by the proposition \ref{exchange_prop}. However, the number of spikes has a thicker distribution tail. The maximum number of presynaptic spikes observed on the $N=100000$ replicates is equal to $42$. The geometric distribution returns this realization with a probability equals to $1.34 \times 10^{-8}$ much lower than $\frac{1}{N}$. This is the inhibition effect. When the neuron $0$ receives many presynaptic spikes, its inhibition increases. It tends to activate less than its neighbors and receives even more presynaptic spikes before resetting.}
        \label{fig: n_jumps}
    \end{figure}
    
    \begin{figure}[h!]
        \centering
        \includegraphics[scale=0.52, height=5cm, width=7cm]{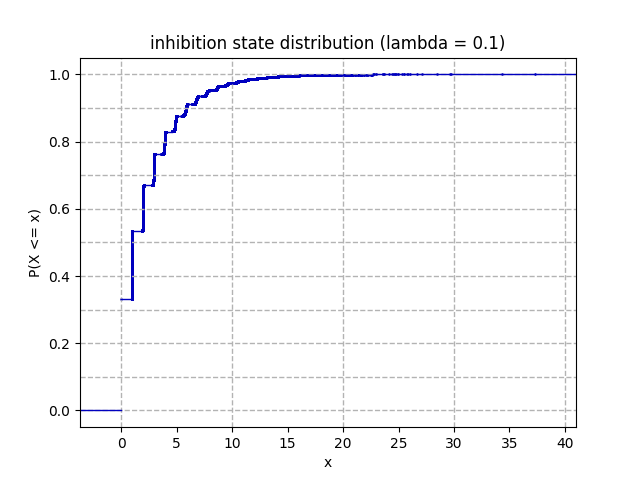}
        \includegraphics[scale=0.52, height=5cm, width=7cm]{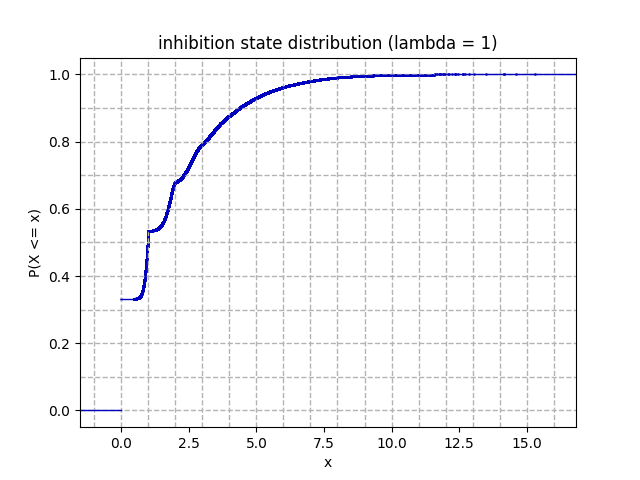}
        \caption{}
        \label{fig: empirical_dist_lamb}
    \end{figure}
    The interaction function $h$ is involved in the short-term memory loss. The parameter $\lambda$ regulates how quickly the neuron loses its memory when it is not firing. If $\lambda$ is small, the inhibition state provides information almost equivalent to the number of presynaptic spikes. With $\lambda = 0.1$, the probable inhibition states are between $0$ and $40$ almost like the range of the number of observed interactions. The inhibition state is actually very close to the sum of the counting processes associated with presynaptic neurons. On the other hand, the inhibition state gives very little information about the timing of past interactions. The larger the value of $\lambda$, the more the interactions are modulated according to their distance with time $0$. On the other hand, it becomes impossible to determine whether the high inhibition states are due to many distant interactions or fewer interactions close to the time $0$. This is what provides regularity to the distribution for high inhibition states. \\

\textbf{Acknowledgments:}\\
We thank Eva L\"{o}cherbach for the stimulating discussions on Hawkes processes and chains with long memory. This research was conducted within the part of the Labex MME-DII(ANR11-LBX-0023-01)  project and the CY Initiative of Excellence (grant "Investissements d'Avenir" ANR-16-IDEX-0008), Project EcoDep PSI-AAP 202-00000013

\end{document}